\documentclass[oneside]{amsart}

\usepackage{color}
\usepackage{amssymb}
\usepackage[a4paper]{geometry}
\usepackage{dsfont}
\usepackage{mathrsfs}
\usepackage[hidelinks]{hyperref}

\theoremstyle{plain}
\newtheorem{thm}{Theorem}[section] 
\newtheorem{prop}[thm]{Proposition} 
\newtheorem{lem}[thm]{Lemma}

\theoremstyle{remark}

\theoremstyle{definition}

\numberwithin{equation}{section}

\begin{document}

\title[Bourgain--Brezis selection maps in Besov and Triebel--Lizorkin spaces]{Bourgain--Brezis selection maps for Hodge systems in Besov and Triebel--Lizorkin spaces}

\author{Diogo Ars\'enio}
\address{New York University Abu Dhabi \\
Abu Dhabi \\
United Arab Emirates} 
\email{\href{mailto:diogo.arsenio@nyu.edu}{diogo.arsenio@nyu.edu}}

\date{November 27, 2026}

\begin{abstract}
	We solve the longstanding Bourgain--Brezis problem for Hodge systems in substantial new ranges of critical Besov and Triebel--Lizorkin spaces. More precisely, we construct Bourgain--Brezis selection maps in Besov spaces $\dot B^{n/p}_{p,q}(\mathbb{R}^n)$ for every $1\leq q\leq p\leq 2$, and in Triebel--Lizorkin spaces $\dot F^{n/p}_{p,q}(\mathbb{R}^n)$ for every $1<p\leq q\leq 2$. These results hold in every dimension $n\geq 2$ and every form degree $1\leq l\leq n-1$, considerably extending the known range of solvability. Our approach combines the trace-free Beurling--Ahlfors transform with a nonlinear duality method and draws on concentration-control techniques developed by Stolyarov in the study of Maz'ya's $\Phi$-inequalities.
\end{abstract}

\maketitle


\section{Introduction and main results}

Consider the critical Besov space
\begin{equation*}
	\dot B^{\frac np}_{p,q}(\mathbb{R}^n)
\end{equation*}
and the critical Triebel--Lizorkin space
\begin{equation*}
	\dot F^{\frac np}_{p,q}(\mathbb{R}^n),
\end{equation*}
with parameters $1<p,q<\infty$ and $n\geq 1$. The corresponding critical Sobolev space is identified with
\begin{equation*}
	\dot W^{\frac np,p}(\mathbb{R}^n)=\dot F^{\frac np}_{p,2}(\mathbb{R}^n),
\end{equation*}
with equivalent norms.

For a normed space $X$, denote the space of differential $l$-forms with coefficients in $X$ by $X\Lambda^l$, $0\leq l\leq n$, and the exterior derivative on $l$-forms by $d$. The expression $A\lesssim B$ is used to denote an inequality $A\leq CB$ for a constant $C>0$ depending only on fixed parameters, and $A\sim B$ means both $A\lesssim B$ and $B\lesssim A$. All relevant functional spaces and notations are introduced with precision in Section~\ref{sec:preliminaries}.

\medskip

Let $n\geq 2$ and $1\leq l\leq n-1$. The Bourgain--Brezis problem for Hodge systems in Besov spaces asks whether every exact datum $dv$ with a primitive $v$ in $\dot B^{n/p}_{p,q}\Lambda^l$ has another primitive $u$ in $(\dot B^{n/p}_{p,q}\cap L^\infty)\Lambda^l$ such that $du=dv$ and
\begin{equation*}
	\|u\|_{(\dot B^{n/p}_{p,q}\cap L^\infty)\Lambda^l}
	\lesssim
	\|v\|_{\dot B^{n/p}_{p,q}\Lambda^l}.
\end{equation*}
An affirmative answer to this question yields a nonlinear selection map $v\mapsto u$ that is bounded from $\dot B^{n/p}_{p,q}$ into $\dot B^{n/p}_{p,q}\cap L^\infty$. The corresponding problem in Triebel--Lizorkin spaces is obtained by replacing $\dot B^{n/p}_{p,q}$ with $\dot F^{n/p}_{p,q}$.

Equivalently, the Bourgain--Brezis problem asks whether
\begin{equation*}
	d\big[\dot B^{\frac np}_{p,q}\Lambda^l\big]
	=
	d\big[(\dot B^{\frac np}_{p,q}\cap L^\infty)\Lambda^l\big],
	\qquad
	d\big[\dot F^{\frac np}_{p,q}\Lambda^l\big]
	=
	d\big[(\dot F^{\frac np}_{p,q}\cap L^\infty)\Lambda^l\big],
\end{equation*}
with equivalent induced norms. The subtlety of this problem lies in the failure of the critical Sobolev embeddings
\begin{equation*}
	\dot B^{\frac np}_{p,q}(\mathbb{R}^n)\not\subset L^\infty(\mathbb{R}^n),
	\qquad
	\dot F^{\frac np}_{p,q}(\mathbb{R}^n)\not\subset L^\infty(\mathbb{R}^n).
\end{equation*}

By devising a constructive approximation procedure, Bourgain and Brezis established \cite{bourgain:brezis:2003,bourgain:brezis:2007} the existence of a bounded selection map in the Sobolev space $\dot W^{1,n}\Lambda^l$, for every form degree $1\leq l\leq n-1$. Then, Bousquet--Mironescu--Russ \cite{bousquet:mironescu:russ:2013} and later Bousquet--Russ--Wang--Yung \cite{bousquet:russ:wang:yung:2019} refined the original Bourgain--Brezis method to construct a selection map in any Triebel--Lizorkin space $\dot F^{n/p}_{p,q}\Lambda^l$, with the constraint $n<p+l$.

Apart from Maz'ya's solution \cite{maz'ya:2010} in the special case $\dot W^{n/2,2}\Lambda^1$, the range $n\geq p+l$ remained out of reach. Recently, we overcame this barrier by proving \cite{arsenio:2026} the existence of a selection map in Sobolev spaces $\dot W^{n/p,p}\Lambda^l$ for a sequence of parameters $p>1$ converging to $1$, in every dimension and for every form degree. We also established corresponding results in Besov spaces $\dot B^{n/p}_{p,q}\Lambda^l$ and Triebel--Lizorkin spaces $\dot F^{n/p}_{p,q}\Lambda^l$ for specific sequences $p\to 1$ and suitable values of $q$.

These results are obtained by using the dual approach to the Bourgain--Brezis problem, which was previously thought to have limited scope, and by exploiting the intrinsic cancellation properties of the trace-free Beurling--Ahlfors transform
\begin{equation}\label{def:AB_transform}
	S=P-\frac ln\operatorname{Id},
\end{equation}
where $P$ denotes the Hodge projection onto exact differential forms.

Our goal now is to solve the Bourgain--Brezis problem throughout continuous parameter ranges beyond the classical restriction $n<p+l$. We focus first on the Besov setting $\dot B^{n/p}_{p,q}\Lambda^l$. Our first result establishes the existence of a Bourgain--Brezis selection in the whole range $1\leq q\leq p\leq 2$.

\begin{thm}[Besov selection maps]\label{thm:Besov_solution}
	Let $n\geq2$, $1\leq l\leq n-1$, and $1\leq q\leq p\leq 2$. Then, for every $v\in\dot B^{n/p}_{p,q}\Lambda^l$, there exists $u\in(\dot B^{n/p}_{p,q}\cap L^\infty)\Lambda^l$ such that $du=dv$ and
	\begin{equation*}
		\|u\|_{(\dot B^{\frac np}_{p,q}\cap L^\infty)\Lambda^l}
		\lesssim
		\|v\|_{\dot B^{\frac np}_{p,q}\Lambda^l}.
	\end{equation*}
	In particular,
	\begin{equation*}
		d\big[\dot B^{\frac np}_{p,q}\Lambda^l\big]
		=
		d\big[(\dot B^{\frac np}_{p,q}\cap L^\infty)\Lambda^l\big]
	\end{equation*}
	with equivalent norms.
\end{thm}

The corresponding result in the critical Triebel--Lizorkin setting $\dot F^{n/p}_{p,q}\Lambda^l$ holds in the range $1< p\leq q\leq 2$.

\begin{thm}[Triebel--Lizorkin selection maps]\label{thm:TL_solution}
	Let $n\geq2$, $1\leq l\leq n-1$, and $1< p\leq q\leq 2$. Then, for every $v\in\dot F^{n/p}_{p,q}\Lambda^l$, there exists $u\in(\dot F^{n/p}_{p,q}\cap L^\infty)\Lambda^l$ such that $du=dv$ and
	\begin{equation*}
		\|u\|_{(\dot F^{\frac np}_{p,q}\cap L^\infty)\Lambda^l}
		\lesssim
		\|v\|_{\dot F^{\frac np}_{p,q}\Lambda^l}.
	\end{equation*}
	In particular,
	\begin{equation*}
		d\big[\dot F^{\frac np}_{p,q}\Lambda^l\big]
		=
		d\big[(\dot F^{\frac np}_{p,q}\cap L^\infty)\Lambda^l\big]
	\end{equation*}
	with equivalent norms.
\end{thm}

Following the methods of \cite{arsenio:2026}, we obtain both theorems as direct consequences of the nonlinear cancellation estimate in the next proposition and of the properties of the trace-free Beurling--Ahlfors transform. The proofs are given in Section~\ref{sec:proofs_main_thm}.

The projectors $\Delta_j$ denote the dyadic operators of a smooth homogeneous Littlewood--Paley decomposition in $\mathbb{R}^n$. They are introduced in detail in Section~\ref{sec:preliminaries}.

\begin{prop}[Nonlinear cancellation]\label{prop:nonlinear_cancellation}
	Let $m\in C^\infty(\mathbb R^n\setminus\{0\})$ be positively homogeneous of degree zero. Assume that
	\begin{equation}\label{eq:cancellation}
		\int_{\mathbb{S}^{n-1}}m(\sigma)d\sigma=0.
	\end{equation}
	Let $M\geq 2$ and $K\geq 1$, and set
	\begin{equation*}
		L=MK-1.
	\end{equation*}
	Then, for all $f,g\in L^1(\mathbb R^n)$,
	\begin{equation}\label{estimate:nonlinear}
		\sum_{j\in\mathbb Z}2^{-j\frac{nL}K}
		\bigg|
		\int_{\mathbb R^n}
		\big(\Delta_jm(D)f\big)(\Delta_jg)|\Delta_jg|^{M-2}
		\bigg(\sum_{i\leq j}2^{-i\frac{nL}K}|\Delta_ig|^M\bigg)^{K-1}
		dx
		\bigg|
		\lesssim
		\|f\|_{L^1}\|g\|_{L^1}^L.
	\end{equation}
\end{prop}

\noindent When $M$ and $K$ are integers, an equivalent version of the nonlinear cancellation estimate \eqref{estimate:nonlinear} was established in \cite{arsenio:2026} using a multilinear argument. Proposition~\ref{prop:nonlinear_cancellation} naturally extends that estimate and provides the key new input that enables the broader parameter ranges in Theorems~\ref{thm:Besov_solution} and ~\ref{thm:TL_solution}. To prove this extension, we draw on the analytical machinery developed by Stolyarov in \cite{stolyarov:2024} for proving Maz'ya's $\Phi$-inequalities. These inequalities capture subtle cancellations of regularized concentrations in integrands that are generally not integrable. This is precisely the phenomenon needed here.

The case $K=1$ reduces to
\begin{equation*}
	\sum_{j\in\mathbb Z}2^{-jn(M-1)}
	\bigg|
	\int_{\mathbb R^n}
	\big(\Delta_jm(D)f\big)(\Delta_jg)|\Delta_jg|^{M-2}dx
	\bigg|
	\lesssim
	\|f\|_{L^1}\|g\|_{L^1}^{M-1}.
\end{equation*}
This simpler version of the nonlinear cancellation estimate is used in the Besov setting. The full multiscale estimate \eqref{estimate:nonlinear} is needed in the Triebel--Lizorkin setting.

\medskip

The paper is organized as follows. Section~\ref{sec:preliminaries} introduces the notation and basic facts concerning homogeneous Littlewood--Paley spaces, differential forms, and Hodge projections. Section~\ref{sec:proofs_main_thm} contains the proofs of the main theorems. The remaining sections establish the nonlinear cancellation estimate \eqref{estimate:nonlinear}. Section~\ref{sec:compact} gives a technical reduction that lets us assume the kernels in \eqref{estimate:nonlinear} have compact support. Section~\ref{sec:concentrations} establishes the key cancellation mechanism for regularized concentrations, building on the methods of Stolyarov \cite{stolyarov:2024}. Finally, Section~\ref{sec:end_proof} combines these estimates to complete the proof of Proposition~\ref{prop:nonlinear_cancellation}.


\section{Preliminaries}\label{sec:preliminaries}

We collect here the notation and some standard facts used throughout this article.

\subsection{Homogeneous functional spaces}

The Schwartz space and the space of tempered distributions are respectively denoted by $\mathscr S(\mathbb R^n)$ and $\mathscr S'(\mathbb R^n)$. The space of polynomials is denoted by $\mathscr P(\mathbb R^n)$.
All homogeneous function spaces are defined here as subspaces of the space of tempered distributions modulo polynomials $\mathscr S'(\mathbb R^n)/\mathscr P(\mathbb R^n)$.

Let $\mathscr S_0(\mathbb R^n)$ denote the subspace of all $\phi\in\mathscr S(\mathbb R^n)$ satisfying
\begin{equation*}
	\int_{\mathbb R^n}x^\alpha\phi(x)dx=0
\end{equation*}
for every multi-index $\alpha$. The space of tempered distributions restricted to $\mathscr{S}_0$ is written $\mathscr{S}_0'(\mathbb{R}^n)$. One can show that $\mathscr{S}'/\mathscr{P}$ and $\mathscr{S}_0'$ are isomorphic.

Fix now a nonnegative radial function $\varphi\in C^\infty_c(\mathbb R^n\setminus\{0\})$ such that
\begin{equation*}
	\sum_{j\in\mathbb Z}\varphi(2^{-j}\xi)=1
	\qquad
	\text{for every $\xi\neq0$},
\end{equation*}
and set
\begin{equation*}
	\Delta_j=\varphi(2^{-j}D)=\mathscr{F}^{-1}\varphi(2^{-j}\xi)\mathscr{F},
\end{equation*}
where we use the convention
\begin{equation*}
	\mathscr{F}f(\xi)=\widehat f(\xi)=\int_{\mathbb{R}^n}e^{-i\xi\cdot x}f(x)dx
\end{equation*}
defining the Fourier transform. The smooth homogeneous Littlewood--Paley decomposition
\begin{equation*}
	u=\sum_{j\in\mathbb{Z}}\Delta_{j}u
\end{equation*}
converges in $\mathscr{S}_0'$.

Let $(s,p,q)\in\mathbb{R}\times[1,\infty]\times[1,\infty]$. The homogeneous Besov space $\dot B^s_{p,q}(\mathbb{R}^n)$ is defined as the subspace of $\mathscr{S}_0'$ endowed with the complete norm
\begin{equation*}
	\left\|u\right\|_{\dot B^s_{p,q}\left(\mathbb{R}^n\right)}=
	\left\|2^{js}\Delta_j u\right\|_{\ell^q L^p}
\end{equation*}
where the $\ell^q$ norm acts on $j\in\mathbb{Z}$. If $p,q\in [1,\infty)$, it can be shown that $\mathscr{S}_0$ is dense in $\dot B^s_{p,q}$.

For $(s,p,q)\in\mathbb{R}\times[1,\infty)\times[1,\infty)$, the homogeneous Triebel--Lizorkin space $\dot F^s_{p,q}(\mathbb{R}^n)$ is defined as the subspace of $\mathscr{S}_0'$ endowed with the complete norm
\begin{equation*}
	\left\|u\right\|_{\dot F^s_{p,q}\left(\mathbb{R}^n\right)}=
	\left\|2^{js}\Delta_j u\right\|_{L^p\ell^q}.
\end{equation*}
The subspace $\mathscr{S}_0$ is dense in $\dot F^s_{p,q}$.

We also recall that, for $p,q\in (1,\infty)$, the standard duality identities are
\begin{equation*}
	\big(\dot B^s_{p,q}\big)'
	=
	\dot B^{-s}_{p',q'},
	\qquad
	\big(\dot F^s_{p,q}\big)'
	=
	\dot F^{-s}_{p',q'},
\end{equation*}
with equivalent norms, where the duality pairing is the extension of the usual integral pairing on $\mathscr{S}_0$.

\subsection{Differential forms and Hodge projections}

For $0\leq l\leq n$, we write $\Lambda^l=\Lambda^l\mathbb{R}^n$ to denote the space of exterior $l$-forms on the Euclidean space $\mathbb{R}^n$. If $X$ is a Banach space, then $X(\mathbb R^n,\Lambda^l)$, abbreviated to $X\Lambda^l$, denotes the space of differential $l$-forms whose coefficients belong to $X$.

The exterior derivative on $l$-forms is denoted by $d$ and its formal $L^2$-adjoint is the codifferential $d^*$. These operators satisfy
\begin{equation*}
	d^2=0,
	\qquad
	(d^*)^2=0,
	\qquad
	dd^*+d^*d=-\Delta,
\end{equation*}
where $\Delta=\sum_{j=1}^n\partial_j^2$ is the classical Laplacian.

The norm induced by $X\Lambda^l$ on $d[X\Lambda^l]$ is given by
\begin{equation*}
	\|dv\|_{d[X\Lambda^l]}
	=
	\inf_{du=dv}\|u\|_{X\Lambda^l}
\end{equation*}
and provides $d[X\Lambda^l]$ with a natural Banach structure.

The Hodge projections are
\begin{equation*}
	P=\frac{dd^*}{|D|^2},
	\qquad
	P^\perp=\frac{d^*d}{|D|^2}.
\end{equation*}
They satisfy
\begin{equation*}
	P^2=P,
	\qquad
	(P^\perp)^2=P^\perp,
	\qquad
	PP^\perp=P^\perp P=0,
	\qquad
	P+P^\perp=\operatorname{Id}.
\end{equation*}

Finally, note that the exterior derivative
\begin{equation*}
	d:\dot B^s_{p,q}\Lambda^l\to \dot B^{s-1}_{p,q}\Lambda^{l+1},
	\qquad
	d:\dot F^s_{p,q}\Lambda^l\to \dot F^{s-1}_{p,q}\Lambda^{l+1},
\end{equation*}
and the codifferential
\begin{equation*}
	d^*:\dot B^s_{p,q}\Lambda^l\to \dot B^{s-1}_{p,q}\Lambda^{l-1},
	\qquad
	d^*:\dot F^s_{p,q}\Lambda^l\to \dot F^{s-1}_{p,q}\Lambda^{l-1},
\end{equation*}
act as bounded operators.


\section{Proofs of main theorems}
\label{sec:proofs_main_thm}

We now show how the main theorems follow from Proposition~\ref{prop:nonlinear_cancellation}, using the approach of \cite{arsenio:2026} based on the trace-free Beurling--Ahlfors transform.

For convenience, introduce the regularity parameter
\begin{equation*}
	s_r=-n\Big(1-\frac1r\Big),
\end{equation*}
for any $1\leq r<\infty$. We denote by $L^1_0(\mathbb{R}^n)$ the mean-zero subspace of $L^1(\mathbb{R}^n)$.

We begin with the proof of existence of a Bourgain--Brezis selection map in Besov spaces.

\begin{proof}[Proof of Theorem~\ref{thm:Besov_solution}]
	Let $2\leq r\leq t<\infty$ and take
	\begin{equation*}
		\omega=\sum_{|I|=l}\omega_I dx_I\in L^1_0\Lambda^l,
	\end{equation*}
	such that $P\omega\in\dot B^{s_r}_{r,t}\Lambda^l$, where the summation runs over all ordered $l$-tuples
	\begin{equation*}
		I=(i_1<i_2<\ldots<i_l),
	\end{equation*}
	with $\{i_1,i_2,\ldots,i_l\}\subset\{1,2,\ldots,n\}$ and
	\begin{equation*}
		dx_I=dx_{i_1}\wedge dx_{i_2}\wedge\ldots\wedge dx_{i_l}.
	\end{equation*}
	By the classical regularization procedure described in \cite{arsenio:2026}, it suffices to assume that the Fourier transform of $\omega$ has compact support disjoint from a neighborhood of the origin.
	
	Fix a component $I$ and set
	\begin{equation}\label{proof:thm:besov:notation}
		u=\omega_I,
		\qquad
		v=(P\omega)_I.
	\end{equation}
	From \eqref{def:AB_transform},
	\begin{equation}\label{proof:thm:besov:decomposition}
		\Delta_ju
		=
		\frac nl\big(
		\Delta_jv-\Delta_j(S\omega)_I
		\big).
	\end{equation}
	Multiplying by $(\Delta_ju)|\Delta_ju|^{r-2}$ and integrating gives
	\begin{equation*}
		\|\Delta_ju\|_{L^r}^r
		\lesssim
		\left|
		\int_{\mathbb R^n}
		\Delta_j(S\omega)_I(\Delta_ju)|\Delta_ju|^{r-2}dx
		\right|
		+
		\|\Delta_jv\|_{L^r}\|\Delta_ju\|_{L^r}^{r-1},
	\end{equation*}
	which, by Young's inequality, leads to
	\begin{equation*}
		\|\Delta_ju\|_{L^r}^r
		\lesssim
		\left|
		\int_{\mathbb R^n}
		\Delta_j(S\omega)_I(\Delta_ju)|\Delta_ju|^{r-2}dx
		\right|
		+
		\|\Delta_jv\|_{L^r}^r.
	\end{equation*}
	Next, multiplying by $2^{-jn(r-1)}$ and taking the $\ell^{t/r}$ norm leads to
	\begin{equation*}
		\|u\|_{\dot B_{r,t}^{s_r}}^r
		\lesssim
		\sum_{j\in\mathbb{Z}}2^{-jn(r-1)}
		\left|\int_{\mathbb{R}^n}\Delta_j(S\omega)_I(\Delta_ju)|\Delta_ju|^{r-2}dx\right|
		+\|v\|_{\dot B_{r,t}^{s_r}}^r,
	\end{equation*}
	where we used that $\ell^1\subset\ell^{t/r}$, for $t\geq r$.
	
	By Proposition~4.4 of \cite{arsenio:2026}, each coefficient of $S\omega$ is a finite linear combination of terms $m(D)\omega_J$, where $m$ satisfies the cancellation hypothesis \eqref{eq:cancellation}. Applying Proposition~\ref{prop:nonlinear_cancellation} with $M=r$ and $K=1$ then yields
	\begin{equation*}
		\|u\|_{\dot B^{s_r}_{r,t}}^r
		\lesssim
		\|\omega\|_{L^1_0\Lambda^l}^r
		+
		\|v\|_{\dot B^{s_r}_{r,t}}^r.
	\end{equation*}
	Summation over the finitely many components of $\omega$ proves that $\omega\in\dot B^{s_r}_{r,t}\Lambda^l$ with
	\begin{equation*}
		\|\omega\|_{\dot B^{s_r}_{r,t}\Lambda^l}
		\lesssim
		\|\omega\|_{L^1_0\Lambda^l}
		+
		\|P\omega\|_{\dot B^{s_r}_{r,t}\Lambda^l}.
	\end{equation*}
	This estimate is dual to the Bourgain--Brezis problem in $\dot B^{n/p}_{p,q}\Lambda^l$.
	
	If $1<q\leq p\leq 2$, the preceding estimate with $r=p'$ and $t=q'$ together with Proposition~3.2 of \cite{arsenio:2026} gives the result. If $q=1$, the conclusion follows directly from the critical embedding
	\begin{equation*}
		\dot B^{\frac np}_{p,1}\subset L^\infty
	\end{equation*}
	by taking $u=v$.
\end{proof}

We next prove the existence of a selection map in Triebel--Lizorkin spaces. We begin with a simple technical lemma.

\begin{lem}\label{lem:technical}
	Let $\beta>1$. There exists $C_\beta>0$ such that
	\begin{equation*}
		\Big(\sum_{j\in\mathbb{Z}}a_j\Big)^\beta
		\leq C_\beta \sum_{j\in\mathbb{Z}}a_j\Big(\sum_{i\leq j}a_i\Big)^{\beta-1}
	\end{equation*}
	for every sequence of nonnegative numbers $(a_j)\in\ell^1(\mathbb{Z})$.
\end{lem}

\begin{proof}
	Choose an integer $k\geq \beta$. By symmetry
	\begin{equation*}
		\Big(\sum_{j\in\mathbb{Z}}a_j\Big)^k
		\leq k \sum_{j_1,\ldots,j_k}a_{j_1}a_{j_2}\cdots a_{j_k}\mathds{1}_{\{j_1=\max\{j_1,j_2,\ldots,j_k\}\}}
		\leq k \sum_{j\in\mathbb{Z}}a_j\Big(\sum_{i\leq j}a_i\Big)^{k-1}.
	\end{equation*}
	Using that
	\begin{equation*}
		\Big(\sum_{i\leq j}a_i\Big)^{k-1}\leq \Big(\sum_{i\leq j}a_i\Big)^{\beta-1}\Big(\sum_{i\in\mathbb{Z}}a_i\Big)^{k-\beta}
	\end{equation*}
	concludes the proof.
\end{proof}

\begin{proof}[Proof of Theorem~\ref{thm:TL_solution}]
	Let $2\leq t\leq r<\infty$. As before, take $\omega\in L^1_0\Lambda^l$ such that $\widehat \omega$ has compact support disjoint from a neighborhood of the origin.
	
	Using the notation introduced in \eqref{proof:thm:besov:notation}, multiply \eqref{proof:thm:besov:decomposition} by
	\begin{equation*}
		2^{jts_r}(\Delta_ju)|\Delta_ju|^{t-2}
		\bigg(\sum_{i\leq j}2^{its_r}|\Delta_iu|^t\bigg)^{\frac rt-1}
	\end{equation*}
	to deduce, by Lemma~\ref{lem:technical}, upon integrating in $x$ and summing over $j$, that
	\begin{equation*}
		\begin{aligned}
			\|u\|_{\dot F^{s_r}_{r,t}}^r
			&\leq
			\frac nl\bigg|\int_{\mathbb{R}^n}\sum_{j\in\mathbb{Z}}
			2^{jts_r}(\Delta_jv)(\Delta_ju)|\Delta_ju|^{t-2}
			\bigg(\sum_{i\leq j}2^{its_r}|\Delta_iu|^t\bigg)^{\frac rt-1}
			dx\bigg|
			\\
			&\quad+\frac nl \sum_{j\in\mathbb{Z}}
			\bigg|\int_{\mathbb{R}^n}
			2^{jts_r}\Delta_j(S\omega)_I(\Delta_ju)|\Delta_ju|^{t-2}
			\bigg(\sum_{i\leq j}2^{its_r}|\Delta_iu|^t\bigg)^{\frac rt-1}
			dx\bigg|.
		\end{aligned}
	\end{equation*}
	By Proposition~4.4 of \cite{arsenio:2026}, each coefficient of $S\omega$ is a finite linear combination of terms $m(D)\omega_J$, where $m$ satisfies the cancellation property \eqref{eq:cancellation}. Applying Proposition~\ref{prop:nonlinear_cancellation} with $M=t$ and $K=r/t$ therefore shows that the last term is controlled by $\|\omega\|_{L^1_0\Lambda^l}^r$.
	
	As for the first term on the right-hand side of the previous estimate, H\"older's inequality shows that it is bounded above by a constant multiple of
	\begin{equation*}
		\bigg|\int_{\mathbb{R}^n}\bigg(\sum_{j\in\mathbb{Z}}
		2^{jts_r}|\Delta_jv||\Delta_ju|^{t-1}\bigg)
		\bigg(\sum_{i\in\mathbb{Z}}2^{its_r}|\Delta_iu|^t\bigg)^{\frac rt-1}
		dx\bigg|
		\leq \|v\|_{\dot F^{s_r}_{r,t}}\|u\|_{\dot F^{s_r}_{r,t}}^{r-1}.
	\end{equation*}
	Hence, by Young's inequality,
	\begin{equation*}
		\|u\|_{\dot F^{s_r}_{r,t}}^r
		\lesssim
		\|\omega\|_{L^1_0\Lambda^l}^r
		+
		\|v\|_{\dot F^{s_r}_{r,t}}^r.
	\end{equation*}
	Summation over the finitely many components of $\omega$ proves that $\omega\in\dot F^{s_r}_{r,t}\Lambda^l$ with
	\begin{equation*}
		\|\omega\|_{\dot F^{s_r}_{r,t}\Lambda^l}
		\lesssim
		\|\omega\|_{L^1_0\Lambda^l}
		+
		\|P\omega\|_{\dot F^{s_r}_{r,t}\Lambda^l}.
	\end{equation*}
	This estimate is dual to the Bourgain--Brezis problem in $\dot F^{n/p}_{p,q}\Lambda^l$.
	
	If $1<p\leq q\leq 2$, the preceding estimate for $r=p'$ and $t=q'$ and Proposition~3.2 of \cite{arsenio:2026} gives the result.
\end{proof}


\section{Reduction to compactly supported kernels}
\label{sec:compact}

Let $A,B\in\mathscr{S}_0(\mathbb{R}^n)$ be the kernels given by
\begin{equation*}
	\mathscr{F}A(\xi)=\varphi(\xi)m(\xi),
	\qquad
	\mathscr{F}B(\xi)=\varphi(\xi),
\end{equation*}
where $m$ is the Fourier multiplier introduced in the statement of Proposition~\ref{prop:nonlinear_cancellation}. Take a radial truncation $\chi\in C_0^\infty(\mathbb{R}^n)$ such that
\begin{equation*}
	\mathds{1}_{\{|x|\leq\frac 12\}}\leq \chi(x)\leq\mathds{1}_{\{|x|\leq 1\}}
\end{equation*}
and define the scaled truncated versions
\begin{equation*}
	\begin{gathered}
		A_j(x)=2^{jn}A(2^jx),
		\qquad
		A_{j,0}(x)=A_j(x)\chi(2^jx),
		\\
		A_{j,N}(x)=A_j(x)\big(\chi(2^{j-N}x)-\chi(2^{j-(N-1)}x)\big),\qquad \text{if }N\geq 1,
	\end{gathered}
\end{equation*}
and
\begin{equation*}
	\begin{gathered}
		B_j(x)=2^{jn}B(2^jx),
		\qquad
		B_{j,N}(x)=B_j(x)\chi(2^{j-(N+1)}x),
		\\
		\widetilde B_{j,N}(x)=B_j(x)(1-\chi)(2^{j-(N+1)}x),
	\end{gathered}
\end{equation*}
for every $j\in\mathbb{Z}$ and $N\in\mathbb{N}$.

By homogeneity of $m$,
\begin{equation}\label{decomposition_double}
	\Delta_jm(D)f=A_j*f=\sum_{N\in\mathbb{N}}A_{j,N}*f,
\end{equation}
and
\begin{equation*}
	\Delta_jg=B_j*g=B_{j,N}*g+\widetilde B_{j,N}*g.
\end{equation*}
Also observe that the cancellation condition \eqref{eq:cancellation} implies
\begin{equation}\label{eq:cancellation:2}
	\int_{\mathbb{S}^{n-1}}A_{j,N}(\sigma)d\sigma=0.
\end{equation}

The next lemma reduces the kernels in the nonlinear cancellation estimate \eqref{estimate:nonlinear} to compactly supported kernels.

\begin{lem}[Reduction to compact supports]\label{lem:compact:support}
	Suppose the assumptions of Proposition~\ref{prop:nonlinear_cancellation} are satisfied. If
	\begin{equation}\label{estimate:nonlinear:compact}
		\begin{aligned}
			\sum_{j\in\mathbb Z}2^{-j\frac{nL}K}
			\bigg|
			\int_{\mathbb R^n}
			(A_{j,N}*f)(B_{j,N}*g)|B_{j,N}*g|^{M-2}
			\bigg(\sum_{i\leq j}2^{-i\frac{nL}K}|B_{i,N+i-j}*g|^M\bigg)^{K-1}
			dx
			\bigg|&
			\\
			\lesssim c_N
			\|f\|_{L^1}\|g\|_{L^1}^L&
		\end{aligned}
	\end{equation}
	holds for a sequence $(c_N)\in\ell^1(\mathbb{N})$, then \eqref{estimate:nonlinear} follows.
\end{lem}

\begin{proof}
	Define $H:\mathbb{R}^2\to\mathbb{R}$ by
	\begin{equation}\label{function:def}
		H(a,b)=a|a|^{M-2}\big(|a|^M+|b|^M\big)^{K-1}
	\end{equation}
	and denote the part of the integrands in \eqref{estimate:nonlinear} and \eqref{estimate:nonlinear:compact} which contains $g$ by
	\begin{equation*}
		\begin{aligned}
			\mathcal{H}_j&=
			H\Big(\Delta_jg,\big\|2^{(j-i)\frac{nL}{MK}}\Delta_ig\big\|_{\ell^M(i<j)}\Big),
			\\
			\mathcal{H}_{j,N}&=
			H\Big(B_{j,N}*g,\big\|2^{(j-i)\frac{nL}{MK}}B_{i,N+i-j}*g\big\|_{\ell^M(i<j)}\Big),
		\end{aligned}
	\end{equation*}
	respectively, where the $\ell^M$ norm acts on the index $i$.
	
	By the mean-value theorem, a direct computation gives
	\begin{equation}\label{function:estimate}
		\big|H(a,b)-H(\widetilde a,\widetilde b)\big|\lesssim
		\big(|a|+|b|+|\widetilde a|+|\widetilde b|\big)^{L-1}\big|(a-\widetilde a,b-\widetilde b)\big|,
	\end{equation}
	for every $(a,b),(\widetilde a,\widetilde b)\in\mathbb{R}^2$.
	This yields
	\begin{equation*}
		\begin{aligned}
			|\mathcal{H}_j-\mathcal{H}_{j,N}|
			&\lesssim
			\Big(\big\|2^{(j-i)\frac{nL}{MK}}\Delta_ig\big\|_{\ell^M(i\leq j)}
			+\big\|2^{(j-i)\frac{nL}{MK}}
			B_{i,N+i-j}*g
			\big\|_{\ell^M(i\leq j)}\Big)^{L-1}
			\\
			&\hspace{70mm}\times
			\big\|2^{(j-i)\frac{nL}{MK}}\widetilde B_{i,N+i-j}*g\big\|_{\ell^M(i\leq j)}
			\\
			&\lesssim
			\Big(\big\|2^{(j-i)\frac{nL}{MK}}B_i
			\big\|_{\ell^M(i\leq j)}*|g|\Big)^{L-1}
			\big(C_{j,N}*|g|\big),
		\end{aligned}
	\end{equation*}
	where
	\begin{equation*}
		C_{j,N}(x)=\big\|2^{(j-i)\frac{nL}{MK}}\widetilde B_{i,N+i-j}(x)\big\|_{\ell^M(i\leq j)}.
	\end{equation*}
	Since
	\begin{equation}\label{technical:1}
		\big\|2^{(j-i)\frac{nL}{MK}}
		B_i
		\big\|_{L^\infty \ell^M(i\leq j)}
		\lesssim 2^{jn}
		\big\|2^{(i-j)\frac{n}{MK}}
		\big\|_{\ell^M(i\leq j)}\|B\|_{L^\infty}\lesssim 2^{jn},
	\end{equation}
	it follows that
	\begin{equation*}
		|\mathcal{H}_j-\mathcal{H}_{j,N}|\lesssim 2^{jn(L-1)}\|g\|_{L^1}^{L-1}\big(C_{j,N}*|g|\big).
	\end{equation*}
	
	Using \eqref{decomposition_double}, this implies that the left-hand side of \eqref{estimate:nonlinear} is bounded above by
	\begin{equation}\label{technical:compact_support}
		\begin{aligned}
				&\sum_{j\in\mathbb Z}\sum_{N\in\mathbb{N}}2^{-jnL}
				\bigg|
				\int_{\mathbb R^n}
				(A_{j,N}*f)\mathcal{H}_jdx
				\bigg|
				\\
				&\lesssim
				\sum_{j\in\mathbb Z}\sum_{N\in\mathbb{N}}
				\left(2^{-jnL}
				\bigg|
				\int_{\mathbb R^n}
				(A_{j,N}*f)\mathcal{H}_{j,N}dx
				\bigg|
				+2^{-jn}\|g\|_{L^1}^{L-1}
				\int_{\mathbb R^n}
				\big(C_{j,N}*|A_{j,N}|*|f|\big)
				|g|dx
				\right).
		\end{aligned}
	\end{equation}
	It remains to bound the last integral. Note that
	\begin{equation*}
		C_{j,N}(x)=2^{jn}C(2^jx)(1-\chi)(2^{j-(N+1)}x),
	\end{equation*}
	with
	\begin{equation*}
		C(x)=\bigg(\sum_{i\leq 0} 2^{i\frac nK}\big|B(2^ix)\big|^M\bigg)^\frac 1M
		\lesssim \big(1+|x|^\frac n{MK}\big)^{-1}.
	\end{equation*}
	Now, the key observations are that $C_{j,N}*|A_{j,N}|(0)=0$, because the two kernels have nonoverlapping supports, and that the slope of
	\begin{equation*}
		2^{-jn}C_{j,N}*|A_{j,N}|(2^{-j}x)=C_{0,N}*|A_{0,N}|(x)
	\end{equation*}
	remains uniformly bounded. Hence,
	\begin{equation}\label{lem:compact:2}
		\sum_{j\in\mathbb{Z}}
		2^{-jn}C_{j,N}*|A_{j,N}|(x)
		\lesssim \sum_{j\in\mathbb{Z}}\frac{|2^jx|}{1+|2^jx|^{1+\frac n{MK}}}\lesssim 1.
	\end{equation}
	We can refine this pointwise estimate using the rapid decay of $A(x)$ and the fact that $2^jx\sim 2^N$ on the support of $A_{j,N}(x)$. Indeed, for any $Q>0$, consider the kernel
	\begin{equation*}
		G_{j,N}(x)=|2^jx|^Q A_{j,N}(x)
	\end{equation*}
	and note that \eqref{lem:compact:2} holds with $G_{j,N}$ in place of $A_{j,N}$. Then,
	\begin{equation*}
		\sum_{j\in\mathbb{Z}}
		2^{-jn}C_{j,N}*|A_{j,N}|(x)
		\lesssim 2^{-NQ}\sum_{j\in\mathbb{Z}}
		2^{-jn}C_{j,N}*|G_{j,N}|(x)
		\lesssim 2^{-NQ}.
	\end{equation*}
	
	Using this pointwise control in the last term of \eqref{technical:compact_support} shows that the left-hand side of \eqref{estimate:nonlinear} is bounded above by
	\begin{equation*}
		\sum_{j\in\mathbb Z}\sum_{N\in\mathbb{N}}
		\left(2^{-jnL}
		\bigg|
		\int_{\mathbb R^n}
		(A_{j,N}*f)\mathcal{H}_{j,N}dx
		\bigg|\right)
		+\|f\|_{L^1}\|g\|_{L^1}^L.
	\end{equation*}
	Combining this bound with \eqref{estimate:nonlinear:compact} concludes the proof.
\end{proof}


\section{Control of concentrations}
\label{sec:concentrations}

This section contains the crux of the argument and draws on methods developed by Stolyarov. The next lemma, analogous to Lemma~3.1 of \cite{stolyarov:2024}, shows that the integral in \eqref{estimate:nonlinear:compact} remains small when the underlying functions concentrate near a single atom. We use the kernel notations from Section~\ref{sec:compact}.

\begin{lem}[Concentration on an atom]\label{lem:concentration}
	Suppose the assumptions of Proposition~\ref{prop:nonlinear_cancellation} are satisfied. Then
	\begin{equation}\label{estimate:single_atom}
		\begin{aligned}
			&\bigg|
			\int_{\mathbb R^n}
			(A_{j,N}*f)(B_{j,N}*g)|B_{j,N}*g|^{M-2}
			\bigg(\sum_{i\leq j}2^{-i\frac{nL}K}|B_{i,N+i-j}*g|^M\bigg)^{K-1}
			dx
			\bigg|
			\\
			&\lesssim 2^{j(1+\frac{nL}K)-\alpha N}
			\|g\|_{L^1}^{L-1}
			\left(\|g\|_{L^1}\int_{\mathbb{R}^n}|f(x)||x-\mu|dx
			+\|f\|_{L^1}\int_{\mathbb{R}^n}|g(x)||x-\mu|dx\right),
		\end{aligned}
	\end{equation}
	for every $\alpha>0$ and $\mu\in\mathbb{R}^n$.
\end{lem}

\begin{proof}
	By scaling invariance, it suffices to treat the case $j=0$. Employing the function $H(a,b)$ given in \eqref{function:def}, define
	\begin{equation*}
		\begin{aligned}
			\mathcal{J}_N&=
			H\Big(B_{0,N}*g(x),
			\big\|2^{-i\frac{nL}{MK}}B_{i,N+i}*g(x)\big\|_{\ell^M(i<0)}
			\Big),
			\\
			\widetilde{\mathcal{J}}_N&=
			H\Big(B_{0,N}(x-\mu)a_g,
			\big\|2^{-i\frac{nL}{MK}}B_{i,N+i}(x-\mu)a_g\big\|_{\ell^M(i<0)}
			\Big),
		\end{aligned}
	\end{equation*}
	where $a_g=\int g(y)dy$ and $\mu\in\mathbb{R}^n$.
	
	Since $A_{0,N}$ satisfies the cancellation property \eqref{eq:cancellation:2}, and $B_{i,N+i}$, for every $i$, is radially symmetric, observe that
	\begin{equation*}
		\int_{\mathbb R^n}
		A_{0,N}(x-\mu)
		\widetilde{\mathcal{J}}_N
		dx=0.
	\end{equation*}
	Hence, the left-hand side of \eqref{estimate:single_atom}, with $j=0$, is bounded above by
	\begin{equation*}
		\bigg|
		\int_{\mathbb R^n}
		\bigg(A_{0,N}*f(x)
		\mathcal{J}_N
		-A_{0,N}(x-\mu)a_f
		\widetilde{\mathcal{J}}_N
		\bigg)
		dx
		\bigg|
		\leq R_1+R_2,
	\end{equation*}
	where $a_f=\int f(y)dy$, and
	\begin{equation*}
		\begin{aligned}
			R_1&=\big\|A_{0,N}*f(x)-A_{0,N}(x-\mu)a_f\big\|_{L^\infty}
			\|\mathcal{J}_N \|_{L^1},
			\\
			R_2&=
			\big\|A_{0,N}(x-\mu)a_f\big\|_{L^1}
			\|\mathcal{J}_N -\widetilde{\mathcal{J}}_N\|_{L^\infty}.
		\end{aligned}
	\end{equation*}
	
	Using \eqref{technical:1}, we bound $R_1$ and $R_2$, respectively, by
	\begin{equation*}
		\begin{aligned}
			R_1&\leq \left\|\int_{\mathbb{R}^n}
			\big|A_{0,N}(x-y)-A_{0,N}(x-\mu)\big|
			|f(y)|dy\right\|_{L^\infty} \big\||B|*|g|\big\|_{L^{M-1}}^{M-1}
			\big\|2^{-i\frac{nL}{MK}}|B_i|*|g|\big\|_{L^\infty \ell^M(i\leq 0)}^{M(K-1)}
			\\
			&\lesssim \|\nabla A_{0,N}\|_{L^\infty}
			\|B\|_{L^{M-1}}^{M-1}\|B\|_{L^\infty}^{M(K-1)}
			\bigg(\int_{\mathbb{R}^n}|f(y)||y-\mu|dy\bigg)
			\|g\|_{L^1}^L
		\end{aligned}
	\end{equation*}
	and, using \eqref{function:estimate},
	\begin{equation*}
		\begin{aligned}
			R_2&\lesssim\|A_{0,N}\|_{L^1}\|f\|_{L^1}
			\left\|2^{-i\frac{nL}{MK}}
			\int_{\mathbb{R}^n}
			\big|B_{i,N+i}(x-y)-B_{i,N+i}(x-\mu)\big|
			|g(y)|dy
			\right\|_{L^\infty \ell^M(i\leq 0)}
			\\
			&\quad\times\big\|2^{-i\frac{nL}{MK}}B_i\big\|_{L^\infty \ell^M(i\leq 0)}^{L-1}\|g\|_{L^1}^{L-1}
			\\
			&\lesssim \|A_{0,N}\|_{L^1}
			\big\|2^{-i\frac{nL}{MK}}\nabla B_{i,N+i}\big\|_{\ell^M(i\leq 0)L^\infty}
			\|B\|_{L^\infty}^{L-1}
			\|f\|_{L^1}\|g\|_{L^1}^{L-1}
			\bigg(\int_{\mathbb{R}^n}|g(y)||y-\mu|dy\bigg).
		\end{aligned}
	\end{equation*}
	Note that
	\begin{equation*}
		\big\|2^{-i\frac{nL}{MK}}\nabla B_{i,N+i}\big\|_{\ell^M(i\leq 0)L^\infty}
		\lesssim \big\|2^{i(1+\frac{n}{MK})}\big\|_{\ell^M(i\leq 0)}\|\nabla B\|_{L^\infty}
		+\big\|2^{i\frac{n}{MK}}\big\|_{\ell^M(i\leq 0)}\|B\|_{L^\infty}
	\end{equation*}
	remains bounded uniformly in $N$.
	
	Finally, for any fixed $\alpha>0$, use that $A$ decays rapidly to majorize the $L^\infty$ norm of $\nabla A_{0,N}$ and the $L^1$ norm of $A_{0,N}$ by a constant multiple of $2^{-\alpha N}$, and reach the conclusion of the proof.
\end{proof}

The preceding estimate can be localized to control concentration on an infinite array of atoms near a lattice. This step is analogous to Theorem~3.4 of \cite{stolyarov:2024}.

For this purpose, consider the grid of dyadic cubes
\begin{equation*}
	Q_{k,a}=\prod_{i=1}^n \big[2^{1-k}a_i,2^{1-k}(a_i+1)\big],
	\qquad k\in\mathbb{Z},
	\qquad a=(a_1,\ldots,a_n)\in\mathbb{Z}^n.
\end{equation*}
Take any cube $Q\subset \mathbb{R}^n$ with side length $l(Q)$. Dilating $Q$ about its center by a factor $\lambda>0$ gives a cube with side length $\lambda l(Q)$, denoted by $\lambda Q$.

\begin{lem}[Concentration on an array of atoms]\label{lem:concentration_lattice}
	Suppose the assumptions of Proposition~\ref{prop:nonlinear_cancellation} are satisfied. Then
	\begin{equation}\label{estimate:lattice_atom}
		\begin{aligned}
			&\bigg|
			\int_{\mathbb R^n}
			(A_{j,N}*f)(B_{j,N}*g)|B_{j,N}*g|^{M-2}
			\bigg(\sum_{i\leq j}2^{-i\frac{nL}K}|B_{i,N+i-j}*g|^M\bigg)^{K-1}
			dx
			\bigg|
			\\
			&\lesssim \sum_{a\in\mathbb{Z}^n}2^{j(1+\frac{nL}K)-\alpha N}
			\|(f,g)\|_{L^1(3Q_{j-N,a})}^L
			\inf_{\mu_a\in\mathbb{R}^n}\int_{3Q_{j-N,a}}\big|(f,g)(x)\big||x-\mu_a|dx,
		\end{aligned}
	\end{equation}
	for every $\alpha>0$.
\end{lem}

\begin{proof}
	As before, by scaling invariance, we only treat the case $j=1$. Set $F=(f,g)$ and denote the integrand in the left-hand side of \eqref{estimate:lattice_atom}, with $j=1$, by $\mathcal{I}(F)$. For any subset $U\subset \mathbb{R}^n$ denote
	\begin{equation*}
		F_U=F\mathds{1}_U.
	\end{equation*}
	For any direction $e\in\mathbb{S}^{n-1}$, consider the partition
	\begin{equation*}
		\mathbb{R}^n=\cup_{k\in\mathbb{Z}}D_k,
		\qquad D_k=\{2^Nk \leq x\cdot e\leq 2^N(k+1) \}.
	\end{equation*}
	
	The kernels $A_{1,N}$ and $B_{i,N+i-1}$, for every $i\leq 1$, are supported in a ball of radius $2^N$ centered at the origin. Thus,
	\begin{equation*}
		\mathcal{I}(F)=\sum_{k\in\mathbb{Z}}\mathds{1}_{D_k}\mathcal{I}(F)=\sum_{k\in\mathbb{Z}}\mathds{1}_{D_k}\mathcal{I}(F_{D_{k-1}\cup D_k\cup D_{k+1}}).
	\end{equation*}
	Observe that $\mathcal{I}(F_{D_{k-1}\cup D_k\cup D_{k+1}})$ is supported on $\cup_{i=k-2}^{k+2}D_i$. Using that
	\begin{equation*}
		\begin{aligned}
			\mathcal{I}(F_{D_{k-1}\cup D_k\cup D_{k+1}})&=\mathcal{I}(F_{D_{k-1}\cup D_k}),
			&\text{on }D_{k-2}\cup D_{k-1},
			\\
			\mathcal{I}(F_{D_{k-1}\cup D_k\cup D_{k+1}})&=\mathcal{I}(F_{D_k\cup D_{k+1}}),
			&\text{on }D_{k+1}\cup D_{k+2},
		\end{aligned}
	\end{equation*}
	it follows that $\mathcal{I}(F)$ can be decomposed into the sum over $k$ of
	\begin{equation*}
		\mathcal{I}(F_{D_{k-1}\cup D_k\cup D_{k+1}})
		-\mathds{1}_{D_{k-2}\cup D_{k-1}}\mathcal{I}(F_{D_{k-1}\cup D_k})
		-\mathds{1}_{D_{k+1}\cup D_{k+2}}\mathcal{I}(F_{D_k\cup D_{k+1}}).
	\end{equation*}
	Hence, summing over $k$ and shifting the summation index by one in the last term,
	\begin{equation*}
		\begin{aligned}
			\mathcal{I}(F)&=
			\sum_{k\in\mathbb{Z}}\big[\mathcal{I}(F_{D_{k-1}\cup D_k\cup D_{k+1}})
			-\mathds{1}_{D_{k-2}\cup D_{k-1}}\mathcal{I}(F_{D_{k-1}\cup D_k})
			-\mathds{1}_{D_{k}\cup D_{k+1}}\mathcal{I}(F_{D_{k-1}\cup D_k})\big]
			\\
			&=
			\sum_{k\in\mathbb{Z}}\big[\mathcal{I}(F_{D_{k-1}\cup D_k\cup D_{k+1}})
			-\mathcal{I}(F_{D_{k-1}\cup D_k})\big].
		\end{aligned}
	\end{equation*}
	
	This procedure can now be iterated $n$ times in each direction of the canonical basis of $\mathbb{R}^n$, which leads to a decomposition
	\begin{equation*}
		\mathcal{I}(F)=\sum_{i\in\mathbb{N}}\varepsilon_i\mathcal{I}(F_{R_i}),
	\end{equation*}
	where $\varepsilon_i\in\{-1,1\}$ and each $R_i$ is a rectangular parallelepiped aligned with the coordinate axes. Each edge of $R_i$ has length $3\cdot 2^N$ or $2^{N+1}$, and each vertex lies on the lattice $2^N\mathbb{Z}^n$. Thus, each $R_i$ is contained inside a cube $3Q_{1-N,a}$, with $a\in\mathbb{Z}^n$. Finally, observing that the family $\{R_i\}_{i\in\mathbb{N}}$ has finite overlap and applying Lemma~\ref{lem:concentration} to each $|\int_{\mathbb{R}^n}\mathcal{I}(F_{R_i})dx|$ concludes the proof.
\end{proof}


\section{Proof of Proposition~\ref{prop:nonlinear_cancellation}}
\label{sec:end_proof}

Using a concentration-control argument analogous to the one in the previous section, Stolyarov showed in \cite{stolyarov:2024} that Maz'ya's $\Phi$-inequalities have a simple proof in the $L^2$ setting. We adopt a similar approach to prove Proposition~\ref{prop:nonlinear_cancellation}.

\begin{proof}[Proof of Proposition~\ref{prop:nonlinear_cancellation}]
	By Lemma~\ref{lem:compact:support}, it remains to prove \eqref{estimate:nonlinear:compact}. Take $\alpha>1$, write $F=(f,g)$, and apply Lemma~\ref{lem:concentration_lattice} to bound the left-hand side of \eqref{estimate:nonlinear:compact} by
	\begin{equation*}
		\begin{aligned}
			&\|F\|_{L^1(\mathbb{R}^n)}^{L-1} \sum_{j\in\mathbb{Z}}\sum_{a\in\mathbb{Z}^n}2^{j-\alpha N}
			\|F\|_{L^1(3Q_{j-N,a})}
			\inf_{\mu_a\in\mathbb{R}^n}\int_{3Q_{j-N,a}}\big|F(x)\big||x-\mu_a|dx
			\\
			&\quad\leq 2^{(1-\alpha) N}
			\|F\|_{L^1(\mathbb{R}^n)}^{L-1} \sum_{j\in\mathbb{Z}}\sum_{a\in\mathbb{Z}^n}2^j
			\int_{3Q_{j,a}\times 3Q_{j,a}}\big|F(x)\big|\big|F(y)\big||x-y|dxdy.
		\end{aligned}
	\end{equation*}
	Since the diameter of $3Q_{j,a}$ is $\sqrt n\cdot 3\cdot 2^{1-j}$, the last line is controlled by
	\begin{equation*}
		2^{(1-\alpha) N}
		\|F\|_{L^1(\mathbb{R}^n)}^{L-1} \sum_{j\in\mathbb{Z}}2^j
		\int_{\{2^j|x-y|\lesssim 1\}}\big|F(x)\big|\big|F(y)\big||x-y|dxdy.
	\end{equation*}
	Thus, observing that
	\begin{equation*}
		\sum_{j\in\mathbb{Z}}2^j|x-y|\mathds{1}_{\{2^j|x-y|\lesssim 1\}}\lesssim 1,
	\end{equation*}
	we have obtained that the left-hand side of \eqref{estimate:nonlinear:compact} is bounded above by
	\begin{equation*}
		2^{(1-\alpha)N}\big(\|f\|_{L^1}^{L+1}+\|g\|_{L^1}^{L+1}\big).
	\end{equation*}
	Now, take $\lambda>0$ and replace $(f,g)$ by $(\lambda^{-L}f,\lambda g)$. Minimizing the right-hand side of the resulting estimate with respect to $\lambda$ proves \eqref{estimate:nonlinear:compact} with $c_N=2^{(1-\alpha)N}$, completing the proof.
\end{proof}


\paragraph{\bf Statement on human intelligence}

The mathematical research in this article is entirely human and AI-free.

\paragraph{\bf Note on a related preprint}

The preprint arXiv:2608.07845 presents a concurrent approach to the Bourgain--Brezis problem. The circumstances of its publication compel me to document the following points:
\begin{itemize}
	\item That preprint was not a source for any result or argument in this article.
	\item It was submitted to arXiv on August 8, 2026, four days after \cite{arsenio:2026} was submitted there on August 4, 2026.
	\item It neither cites nor acknowledges \cite{arsenio:2026}.
	\item It exhibits substantial and specific conceptual, analytical, expository, and structural similarities with \cite{arsenio:2026}, beyond the overlap in the problem addressed.
	\item Its authors disclose that the TARS system assisted their mathematical derivations through exploratory reasoning.
\end{itemize}
Taken together, these circumstances raise a serious concern that an AI system may have drawn on my earlier preprint without proper attribution. The public record does not establish whether this occurred. I do not state it as fact or allege deliberate misconduct by any author.


\bibliographystyle{alpha} 
\bibliography{hodge}

\end{document}